\documentclass[journal,transmag]{IEEEtran}
\usepackage[cmex10]{amsmath}
\usepackage{bbold}
\usepackage{mathrsfs, amsmath, amssymb, bbm}
\usepackage{breqn}

\usepackage{graphicx, pifont} 
\let\oldding\ding% Store old \ding in \oldding
\renewcommand{\ding}[2][1]{\scalebox{#1}{\oldding{#2}}}% Scale \oldding via optional argument

\usepackage{caption}
\usepackage{subfigure}

\newtheorem{theorem}{Theorem}
\newtheorem{lemma}{Lemma}
\newtheorem{proposition}{Proposition}
\newtheorem{proof}{Proof}
\newtheorem{corollary}{Corollary}

\usepackage{pgfplots,pgfplotstable,booktabs}

\pgfplotsset{compat=1.5}

\graphicspath{{images/}}

\usepackage{etoolbox}

\apptocmd{\sloppy}{\hbadness 10000\relax}{}{}

\makeatletter
\newcommand*{\rom}[1]{\expandafter\@slowromancap\romannumeral #1@}
\newcommand{\norm}[1]{\left\lVert#1\right\rVert}
\makeatother
\usepackage{tikz}

\usepackage{cite}

\begin{document}

\title{Contraction Analysis of Nonlinear DAE Systems}


% \author{\IEEEauthorblockN{Hung D. Nguyen \IEEEauthorrefmark{1}, Jean-Jacques Slotine, Konstantin Turitsyn}
% \IEEEauthorblockA{\IEEEauthorrefmark{1}Department of Mechanical Enginering,
% MIT, Cambridge, MA 02139 USA}

\author{\IEEEauthorblockN{Hung D. Nguyen, Thanh Long Vu, Jean-Jacques Slotine, Konstantin Turitsyn}
\IEEEauthorblockA{Department of Mechanical Engineering,
MIT, Cambridge, MA 02139 USA}
\thanks{Corresponding author: Hung D. Nguyen (email: hunghtd@mit.edu).}}

\markboth{IEEE Transactions on Automatic Control, 2016}%
{Shell \MakeLowercase{\textit{et al.}}: IEEE Transactions on Automatic Control}

\IEEEtitleabstractindextext{%
\begin{abstract}
This paper studies the contraction properties of nonlinear differential-algebraic equation (DAE) systems. Specifically we develop scalable techniques for constructing the attraction regions associated with a particular stable equilibrium, by establishing the relation between the contraction rates of the original systems and the corresponding virtual extended systems. We show that for a contracting DAE system, the reduced system always contracts faster than the extended ones; furthermore, there always exists an extension with contraction rate arbitrarily close to that of the original system. The proposed construction technique is illustrated with a power system example in the context of transient stability assessment.

% In this work we establish the connection between contraction analysis for reduced DAE systems and extended systems based on different norms. For the same transform metric of states, the extended system may contract in a slower rate than that of the reduced system. Scalable technique for constructing inner approximation of contraction regions is introduced, based on which an invariant region is characterized by solving an inscribing problem. For simulations we use a power system example.
\end{abstract}

}

\maketitle

\IEEEdisplaynontitleabstractindextext
\IEEEpeerreviewmaketitle

\section{Introduction}
Differential-algebraic equations (DAE)-a generalization of ordinary-differential equations (ODE)-arise in many science and engineering problems, including networks, multibodies, optimal control, compressed fluid, etc. \cite{schulz2003four}. Typically, algebraic constraints result from multiple time-scale perturbation theory, when the fast degrees of freedom are assumed to stay on equilibrium manifold. In typical electrical and mechanical applications the algebraic relations represent the interconnection constraints, which can be considered static on the time-scales of system evolution. However, algebraic relations may be also useful for lifted representations of the purely differential systems. For instance, additional variables and relations can be used to represent any polynomial nonlinearity in a quadratic DAE form. Hence, DAE systems provide a powerful framework for studying nonlinear systems of very general structure. This work is motivated by the DAE representations of the power system models, but the results are presented in a general form. 

The specific problem that motivates our study is the problem of approximating the region of attraction of DAE equilibrium points. The normal operating points of modern power systems lack global stability because of the nonlinearities naturally appearing in these systems. Characterization of the attraction region and more generally assessment of the system security, i.e. its ability to sustain all kinds of faults and disturbances, is an essential task of modern power system operations. As will be shown throughout the paper, the contraction provides a natural framework for constructing the approximations of the attraction region for a broad range of nonlinear DAE problems, such as those arising in power systems. 

Transient stability analysis is a common engineering procedure referring to the ability of the system to converge to a stable post-fault equilibrium after being subject to disturbances. The incremental stability introduced in \cite{lohmiller1998contraction} suggests an alternative way to look at the convergence of the post-fault trajectories. In the light of contraction theory, the virtual displacements of the states tend to zero as the time goes to infinity, or in other words, all the trajectories shrink and converge to the nominal one. Contraction analysis becomes a powerful tool for nonlinear analysis and control \cite{slotine2003modular,lohmiller1998contraction,slotine1991applied,del2013contraction}. The key property of the contraction is the preservation under different system combinations, which is advantageous in network analysis.
%angeli2002lyapunov
% Analyzing the unreduced DAE systems may overcome the restriction resulted from eliminating algebraic variables.

In this paper we focus on the contraction analysis for nonlinear DAE systems. Specifically we develop a practical way of constructing the attraction regions by determining the relation between the contraction rates of the original DAE systems and its extension to virtual dynamics in differential-algebraic space. The extended system can be thought of as a virtual differential system that reduces to a given DAE after the restriction of a subset of variables to their equilibrium manifold. There can be multiple extensions of a given DAE system, each characterized by different contraction rates. However, we show that the contraction rate of the reduced system is always higher, and on the other hand, there always exists an extension with a contraction rate arbitrarily close to the original DAE system. Our results hold for the most commonly used $1$, $2$, and $\infty$ norms, but can likely be extended to more general cases.  We use the theoretical results to develop a scalable technique for constructing ellipsoidal inner approximations of contraction regions from the $2$ norm contraction metric. We illustrate the technique with a practical example from power systems.

%The algebraic relations indeed dictate the key difference in the properties of the solutions to the two representations. Moreover, these constraints are also different depending on the context. For constrained mechanic problems, the algebraic equations describe the performance specifications rather than physical relations. In contrast, those of networks represent the conservation laws; for instance, the power conservation laws in electrical circuits.
\section{Main results} \label{sec:main}
% The general DAE has the following:
% \begin{equation}
%     0 = \mathcal{F}(t,\mathbf{x},\dot{\mathbf{x}})
% \end{equation}
% where if $\frac{\partial{\mathcal{F}}}{\partial{\dot{\mathbf{x}}}}$ is singular and no equivalent ODE exists \cite{PETZOLD1992269}. 
As motivated by the dynamics of electrical power systems, we constrain ourselves to semi-explicit index $1$ structural form as below:
\begin{align} \label{eq:dif}
 \dot{\mathbf{x}} &= f(\mathbf{x},\mathbf{y}), \\
 0 & = g(\mathbf{x},\mathbf{y}). \label{eq:alg}
\end{align}
In this representation, vector $\mathbf{x} \in \mathbb{R}^n$ corresponds to dynamic state variables, $\mathbf{y} \in \mathbb{R}^m$ refers to algebraic variables (whose dynamics is assumed to be fast/instantaneous relative to the dynamics of the state variables). For this class of systems, it is impossible to obtain equivalent ODEs.

For convenience, reduction techniques are widely used to eliminate the algebraic variables. Yet this practice may prohibit one from exploring the underlying structure of the DAE form. To that end a number of works in the literature concentrate on the original systems rather than the reduced ones, for instant, in the context of stability analysis of the descriptor form as below:
\begin{equation} \label{eq:noneq}
 E \dot{\mathbf{z}} = h(\mathbf{z}),
\end{equation}
with $\mathbf{z}^T = [\mathbf{x}^T, \mathbf{y}^T]$, $h^T = [f^T, g^T]$ and $E$ being a diagonal $\mathbb{R}^{(n+m)\times(n+m)}$ matrix with $E_{ii} = 1 $ for $i \leq n$ and $E_{ii} = 0$ otherwise \cite{Lewis1985, bond2009stable, wu1994stability,MASUBUCHI1997669}.

For any given differential state $\mathbf{x}$ the equation \eqref{eq:alg} may have multiple or no solutions for $\mathbf{y}$. In engineering and natural systems that motivate this study, disappearance of all the solutions is usually an indicator of inappropriate modeling that should be fixed accordingly, typically by introducing the fast dynamics of the algebraic states in the model. We don't consider this scenario in our work, and we assume that for every $\mathbf{x}$ there exists at least one solution $\mathbf{Y}(\mathbf{x})$ of the algebraic system of equations \eqref{eq:alg}. For every solution branch we can naturally define the domain $\mathbf{x} \in \mathcal{R}$ where such a solution exists and can be tracked via homotopy/continuation procedure. This domain is characterized by non-singularity of the algebraic Jacobian: 
\begin{equation}
    \mathcal{R} =\left\{\mathbf{x}:\,\,\mathrm{det}\left(\frac{\partial \mathbf{g}}{\partial\mathbf{y}}\Big|_{\mathbf{y} = \mathbf{Y}(\mathbf{x})}\right)\neq 0\right\}.
\end{equation}
We restrict our analysis only to such a domain associated with a specific solution branch. For a system of differential-algebraic equations \eqref{eq:dif}, \eqref{eq:alg} we introduce the Jacobian  defined as
\begin{equation}
    J(\mathbf{x}, \mathbf{y}) = \begin{bmatrix} \partial \mathbf{f}/\partial\mathbf{x} & \partial \mathbf{f}/\partial\mathbf{y} \\
    \partial \mathbf{g}/\partial\mathbf{x} & \partial \mathbf{g}/\partial\mathbf{y}
    \end{bmatrix}.
\end{equation}
To simplify the notations we also define its restriction to the algebraic manifold \eqref{eq:alg} as follows:
\begin{equation}
    %J_\star(\mathbf{x}) = 
    J(\mathbf{x}, \mathbf{Y}(\mathbf{x})) = 
    \begin{bmatrix}
    A & B \\
    C & D
    \end{bmatrix}.
\end{equation}
One of the primary goals of this paper is to provide a characterization of the contraction and invariant regions in the state space of a DAE system. We formally define the contraction domains $\mathcal{C}_p$ as set of differential states $\mathbf{x}$ for which there exists an invertible metric $\theta(\mathbf{x}) \in \mathbb{R}^{n \times n}$ such that the differential equation $\dot{\mathbf{x}} = \mathbf{f}(\mathbf{x},\mathbf{Y}(\mathbf{x}))$ is locally contracting with respect to this metric with some rate $\beta > 0$. Given that for any infinitesimal displacement $\delta \mathbf{x}$ we have $\delta \mathbf{y} = - D^{-1}C \delta \mathbf{x}$ the standard contraction arguments presented in \cite{slotine1991applied,lohmiller1998contraction} lead to the following Proposition.
% definition of the contraction region reduces to
\begin{proposition} \label{prop:1}
The DAE system \eqref{eq:dif} \eqref{eq:alg} is contracting with respect to the metric $\theta(\mathbf{x})$ in the domain $\mathcal{C}_p$ if for all $\mathbf{x} \in \mathcal{C}_{p}$ one has $\mu_p\left( F_r\right) \leq -\beta$ with some $\beta >0 $ and 
\begin{align}
    F_r &=  \dot{\theta}\theta^{-1} + \theta (A-BD^{-1}C)\theta^{-1}. \label{eq:Fr}
\end{align}
\end{proposition}
The term $\dot\theta$ in \eqref{eq:Fr} represents the derivative of the metric along the trajectory and is formally defined for DAE systems as
\begin{equation} \label{eq:theta}
    \dot\theta = %\frac{\partial\theta}{\partial t} +
    \left(\frac{\partial\theta}{\partial \mathbf{x}} \right)^T \mathbf{f}(\mathbf{x}, \mathbf{Y}(\mathbf{x})).
\end{equation}
The matrix measure $\mu_p(M)$ of a matrix $M$ is defined as $\mu_p(M) := \lim\limits_{h\rightarrow0^+}\frac{1}{h}(||\mathbb{1}+h M||_p - 1)$ following \cite{VIDYASAGAR197890}. The standard matrix measures as well as vector norms are listed in Table \ref{table:matmu}.
\begin{table}[t]
\centering
\begin{tabular}{ |c|c| } 
 \hline
 Vector norm, $\norm{\cdot}$ & Matrix measure, $\mu_p(M)$ \\ 
 \hline
 $\norm x_1 = \sum_{i}|x_i|$ & $\mu_1(M)=\max_j(m_{jj}+\sum_{i\neq j}|m_{ij}|)$ \\ 
  \hline
 $\norm x_2 = (\sum_{i}|x_i|^2)^{1/2}$ & $\mu_2(M)=\max_i(\lambda_i\{\frac{M+M^T}{2}\})$ \\ 
 \hline
 $\norm x_\infty = \max_{i}|x_i|$ & $\mu_\infty(M)=\max_i(m_{ii}+\sum_{j\neq i}|m_{ij}|)$\\
 \hline
\end{tabular}
\caption{Standard matrix measures}
\label{table:matmu}
\end{table}
The proof of proposition \ref{prop:1} directly follows from the contraction analysis for dynamical system presented in \cite{lohmiller1998contraction}. The matrix $F_r$ appears naturally from the dynamic equation on $\delta \mathbf{v} = \theta \delta\mathbf{x}$ given by  $\dot{\delta \mathbf{v}} = F_r \delta\mathbf{v}$. Hereafter we refer to $F_r$ as the \textit{generalized reduced Jacobian matrix}. 

The standard contraction theory arguments suggest that for any two trajectories $\mathbf{x}_1(t), \mathbf{x}_2(t)$ that both remain within the contraction region $\mathcal{C}_{p}$ during the interval $[t_1,t_2]$ satisfy $d(\mathbf{x}_1(t_2), \mathbf{x}_2(t_2)) \leq d(\mathbf{x}_1(t_1), \mathbf{x}_2(t_1))\exp(-\beta (t_2-t_1))  $ where $d$ is the distance associated with the metric $\theta$. The assumption that both of the trajectories stay within the contraction region is critical for this result and can be verified only after showing the existence of an invariant domain $\mathcal{I}_{p} \subset \mathcal{C}_{p}$ satisfying:
\begin{equation}
    \mathbf{x}(t) \in \mathcal{I}_{p} \Longrightarrow 
    \forall t' \geq t: \quad  \mathbf{x}(t') \in \mathcal{I}_{p}.
\end{equation}
Constructing invariant regions is usually a difficult aspect of applying contraction theory to systems which are not globally contracting. One straightforward strategy for constructing invariant regions exists for systems that have an equilibrium point $\mathbf{x}_\star$ inside the contraction domain satisfying $\mathbf{f}(\mathbf{x}_\star,\mathbf{Y}(\mathbf{x}_\star)) = 0$. In this case, any ball $\mathcal{B}_r= \{\mathbf{x}: d(\mathbf{x},\mathbf{x}_\star) \leq r\}$ that lies within the contraction region $\mathcal{C}_{p}$ defines an invariant region, i.e.
$\mathcal{B}_r \subset \mathcal{C}_{p} \Longrightarrow \mathcal{B}_r \subset \mathcal{I}_{p}$.
By construction, such a ball also provides an inner approximation for the attraction region of $\mathbf{x}_\star$ and can be naturally used in a variety of practical applications such as security assessment of power systems \cite{HungTuritsyn2015Robust}. In this we develop a general framework for constructing such invariant regions for a broad class of nonlinearities, and we present a specific power system example in sections \ref{sec:conreg}.

The key challenge in using the function $F_r$ directly is its highly nonlinear nature. Even for simple polynomial nonlinearities of $\mathbf{f},\mathbf{g}$ the function $F_r$ involves an inversion of the matrix $D(\mathbf{x})$. From a practical perspective, it is therefore desirable to formulate conditions equivalent to contraction as defined in Propostion \ref{prop:1} that do not involve any inversions of matrices $A, B, C, D$ which are nonlinearly dependent on $\mathbf{x}$. In order to achieve this goal we derive equivalent representation of the contraction condition that doesn't require elimination of the local variables and is more suitable for analysis. We introduce the \textit{generalized unreduced Jacobian matrix} as follows:
\begin{align} 
   F 
     &= \begin{bmatrix} F_r + \theta R^T C \theta^{-1} & \theta R^T D \rho^{-1} \\
   Q^T C \theta^{-1} & Q^T D \rho^{-1}\end{bmatrix}.\label{eq:F}
 \end{align}
The generalized unreduced Jacobian $F$ depends on the metric $\theta$ defined as in the previous discussion, another metric $\rho$ associated with the $\mathbf{y}$ variable and two auxiliary matrices $Q \in \mathbb{R}^{m\times m}$ and $R \in \mathbb{R}^{m\times n}$. Formally, this Jacobian matrix may be associated with a virtual extended ODE representation of the original system of the form
\begin{align}
 \dot{\delta \mathbf{v}} 
 &= F_r \delta\mathbf{v} + \theta R^T(C\theta^{-1}\delta \mathbf{v} + D\rho^{-1}\delta \mathbf{u}), \\
 \dot{\delta \mathbf{u}}  &= Q^T(C\theta^{-1}\delta \mathbf{v} + D\rho^{-1}\delta \mathbf{u}).
\end{align}
where $\delta \mathbf{u} = \rho \delta \mathbf{y}$ and so the expression $C\theta^{-1}\delta \mathbf{v} + D\rho^{-1}\delta \mathbf{u} = C\delta \mathbf{x} + D \delta \mathbf{y} = 0$ defines the algebraic manifold. Whenever the dynamics of $\delta \mathbf{u}$ can be considered fast, the restriction of the $\delta \mathbf{u}$ variables to their equilibrium manifold results in the original DAE systems. Therefore, this representation provides a family of extended representations that reduce to the same original system. It  will be shown that this representation is useful for characterization of the contraction and invariant regions. 

The key property important for the analysis is defined in the following relation:
\begin{align} \label{eq:FFr}
F\delta{\mathbf{w}} &= \begin{bmatrix} F_r + \theta R^T C \theta^{-1} & \theta R^T D \rho^{-1} \\
  Q^T C \theta^{-1} & Q^T D \rho^{-1} \end{bmatrix}\begin{bmatrix}  \theta\delta{\mathbf{x}} \\  \rho\delta{\mathbf{y}} \end{bmatrix}\nonumber \\
& = \begin{bmatrix}F_r\theta \delta{\mathbf{x}} \\ 0 \end{bmatrix} + \begin{bmatrix}
\theta R^T(C \delta \mathbf{x}+ D \delta \mathbf{y}) \\ Q^T(C \delta \mathbf{x} + D \delta \mathbf{y}) \end{bmatrix}.
\end{align}
%the metric tensor $P = \theta^T \theta$ and 
where we have introduced the new variables vector $\delta \mathbf{w} \triangleq \begin{bmatrix} \delta \mathbf{v} \\ \delta \mathbf{u} \end{bmatrix}$. This observation allows us to formulate the following central results of this work.
% As one can see the contraction rates in the differential and the extended space coincide whenever the the system is on the algebraic manifold with $\delta{\mathbf{y}} = -D^{-1} C \delta{\mathbf{x}}$. The second property is that most of the DAEs can be represented in a quadratic form by introducing extra variables, the generalized Jacobian depends linearly in the new variables. The linearity then becomes crucial for scalable techniques of contraction region construction which is discussed in section \ref{sec:conreg}.

\subsection{Forward theorems: from extended systems to reduced ones}
\begin{lemma} \label{lem:1}
Define 

\begingroup\makeatletter\def\f@size{9.5}\check@mathfonts
\def\maketag@@@#1{\hbox{\m@th\large\normalfont#1}}%
\begin{align}
\gamma 
= \biggl\| \begin{bmatrix}\delta \mathbf{v} + h F_r \delta \mathbf{v} \\ \delta \mathbf{u} \end{bmatrix} \biggr\|_p - \biggl\|\begin{bmatrix}\delta \mathbf{v} \\ \delta \mathbf{u} \end{bmatrix}\biggr\|_p \nonumber
- \left(\|\delta \mathbf{v} + h F_r \delta \mathbf{v}\|_p - \|\delta \mathbf{v}\|_p \right)
\end{align}\endgroup

where $h > 0$. Then for all $p \geq 1$, $\gamma \geq 0$ if the following condition holds.

\begin{equation} \label{eq:deltaw}
\biggl\| \begin{bmatrix}\delta \mathbf{v} + h F_r \delta \mathbf{v} \\ \delta \mathbf{u} \end{bmatrix} \biggr\|_p - \biggl\|\begin{bmatrix}\delta \mathbf{v} \\ \delta \mathbf{u} \end{bmatrix}\biggr\|_p \leq 0.
\end{equation}
\end{lemma}
The proof of Lemma \ref{lem:1} is as the following. For fixed $\delta \mathbf{v}$ and $h F_r \delta \mathbf{v}$, $\gamma$ depends solely on $\delta \mathbf{u}$. Taking partial derivative of $\gamma$ with respect to $|\delta \mathbf{u}_j|$  using the definition of $p$ norm for a vector, i.e. $\|\mathbf{v}\|_p = \left( \sum_i{|v_i|^p}\right)^{1/p}$, yields the following:
\begin{align} \label{eq:delta}
\frac{\partial \gamma}{\partial |\delta \mathbf{u}_j|}
= |\delta \mathbf{u}_j|^{p-1} 
&\biggl( \biggl\| \begin{bmatrix}\delta \mathbf{v} + h F_r \delta \mathbf{v} \\ \delta \mathbf{u} \end{bmatrix} \biggr\|_p^{p(1/p -1)}\nonumber\\
&-  \biggl\|\begin{bmatrix}\delta \mathbf{v} \\ \delta \mathbf{u} \end{bmatrix}\biggr\|_p^{p(1/p -1)}\biggr).
\end{align}
% \begin{align}
% \gamma(\delta \mathbf{u}) 
% &= \left(\sum_{i=1}^n|(\delta \mathbf{v}+h F_r \delta \mathbf{v})_i|^p + \sum_{j=n+1}^N |\delta \mathbf{u}_j| \right)^{1/p} -  \left(\sum_{i=1}^n|\delta \mathbf{v}_i|^p + \sum_{j=n+1}^N |\delta \mathbf{u}_j|\right)^{1/p} \nonumber\\
% & - \left(\sum_{i=1}^n|(\delta \mathbf{v}+h F_r \delta \mathbf{v})_i|^p \right)^{1/p} -  \left(\sum_{i=1}^n|\delta \mathbf{v}_i|^p\right)^{1/p}
% \end{align}
On the other hand, the assumption $p \geq 1$ leads to $p(1/p-1) \leq 0$. This together with \eqref{eq:deltaw} and \eqref{eq:delta} concludes $\frac{\partial \gamma}{\partial |\delta \mathbf{u}_j|} \geq 0$ for all $j = 1,\dots,m$. In other words, $\gamma$ is indeed a monotonically increasing function with respect to the absolute value of each entry $\delta \mathbf{u}_j$. Moreover it can be seen that $\gamma$ vanishes when $\delta \mathbf{u} = 0$. Therefore for any non-zero $\delta \mathbf{u}$, $\gamma$ is non-negative.
% and $\biggl\| \begin{bmatrix}\delta \mathbf{v} + h F_r \delta \mathbf{v} \\ \delta \mathbf{u} \end{bmatrix} \biggr\|_p \leq \biggl\|\begin{bmatrix}\delta \mathbf{v} \\ \delta \mathbf{u} \end{bmatrix}\biggr\|_p$. This together with \eqref{eq:delta} 

For infinity norm one can prove the non-negativity of the partial derivatives $\frac{\partial \gamma}{\partial |\delta \mathbf{u}_j|}$ by taking the limit as $p$ goes to infinity and exploiting the assumption presented by \eqref{eq:deltaw}. Alternatively $\gamma$ can be directly evaluated using the matrix measure expression associated with infinity norm listed in Table \ref{table:matmu}.

With Lemma \ref{lem:1} we introduce the first central result as the following.

\begin{theorem} \label{theo:p21}
For the system $\dot{\mathbf{x}} = \mathbf{f}(\mathbf{x},\mathbf{Y}(\mathbf{x}))$ and metric function $\theta(\mathbf{x})$, and contracting extended system $F$ with $\mu_p(F) < 0$ characterized by the matrices $Q,R, \rho$, the following relation holds:
\begin{equation} \label{eq:FrF}
	\mu_p(F_r) \leq \mu_p(F) \nu_p(H).
\end{equation}
in which $S = \rho D^{-1} C \theta^{-1}$, $H = \begin{bmatrix}
    1 \\
    S
    \end{bmatrix}$,
and 
\begin{align}
\nu_p(H) = \min \limits_{\norm{v}_p=1} \norm{Hv}_p.
%=
% \begin{cases}
% 0 & \text{if} \, \mathrm{det}(S^TS) = 0\\
% 1/\norm{S^\dag}_p & \text{if}\, \mathrm{det}(S^TS) \neq 0
% \end{cases}
\end{align}
Note that for invertible $H$, one has $\nu_p(H) = 1/\norm{H^{-1}}_p$.
\end{theorem}

\begin{proof} 
The matrix induced norm definition $\norm{M} = \max \limits_{\norm{v}=1} \norm{Mv}$ implies that for each positive scalar $h$, there exists $\delta \mathbf{v}_h$ with $\norm{\delta \mathbf{v}_h}_p > 0$
satisfying the following equality:
\begin{equation} \label{eq:muv}
\frac{\norm{\delta \mathbf{v}_h + h F_r \delta \mathbf{v}_h}_p - \norm{\delta \mathbf{v}_h}_p}{h \norm{\delta \mathbf{v}_h}_p} = \frac{\norm{\mathbb{1} + h F_r}_p -1}{h}.
\end{equation}
The logarithmic norm is then defined as
\begin{equation}
\mu_p(F_r) = \lim\limits_{h \rightarrow 0}
\frac{\|\delta \mathbf{v}_{h} + h F_r \delta \mathbf{v}_{h}\|_p - \|\delta \mathbf{v}_{h}\|_p}{h \norm{\delta \mathbf{v}_{h}}_p}.
\end{equation}
Lemma 1 implies that this expression can be also rewritten as 
\begin{equation} \label{eq:mpdef}
\mu_p(F_r) \leq \lim\limits_{h \rightarrow 0}
\frac{\biggl\|\delta \mathbf{w}_{h} + h \begin{bmatrix} F_r \delta \mathbf{v}_{h} \\ 0 \end{bmatrix} \biggr\|_p - \|\delta \mathbf{w}_{h} \|_p}{h \norm{\delta \mathbf{v}_{h}}_p}.
\end{equation}
On the other hand, applying the property defined by \eqref{eq:FFr} we have that
 \begin{align} \label{eq:r2e}
 & \delta \mathbf{w}_{h} + h \begin{bmatrix} F_r \delta \mathbf{v}_{h} \\0 \end{bmatrix} \nonumber\\
& = \delta \mathbf{w}_{h} + h F \delta \mathbf{w}_{h} - h\begin{bmatrix}
\theta R^T(C \delta \mathbf{x}_{h}+ D \hat{\delta \mathbf{y}}_{h}) \\ Q^T(C \delta \mathbf{x}_{h}+ D \hat{\delta \mathbf{y}}_{h}) \end{bmatrix} \nonumber\\
&=\delta \mathbf{w}_{h} + h F \delta \mathbf{w}_{h}
\end{align}
where $\delta \mathbf{x}_{h} = \theta^{-1} \delta \mathbf{v}_{h}$ and $\hat{\delta \mathbf{y}}_{h} = -D^{-1}C \delta \mathbf{x}_{h}$. 
Combining \eqref{eq:mpdef} and \eqref{eq:r2e}, we have that
\begin{align}
\mu_p(F_r)
& \leq \lim\limits_{h \rightarrow 0} \frac{\| \delta \mathbf{w}_{h} + h F \delta \mathbf{w}_{h}\|_p -\| \delta \mathbf{w}_{h}\|_p}{h \norm{\delta \mathbf{v}_{h}}_p} \nonumber\\
& = \lim\limits_{h \rightarrow 0} \frac{\| \delta \mathbf{w}_{h} + h F \delta \mathbf{w}_{h}\|_p -\| \delta \mathbf{w}_{h}\|_p}{h \norm{\delta \mathbf{w}_{h}}_p} \frac{\norm{\delta \mathbf{w}_{h}}_p}{\norm{\delta \mathbf{v}_{h}}_p}. \label{eq:pnorm}
% & \leq \lim\limits_{h \rightarrow 0} \frac{\| \delta \mathbf{w}_{h} + h F \delta \mathbf{w}_{h}\|_p -\| \delta \mathbf{w}_{h}\|_p}{h \norm{\delta \mathbf{w}_{h}}_p} \nonumber\\
% & \leq \mu_p(F) \label{eq:FrlF}
\end{align}
By definition 
\begin{align}
\delta \mathbf{w}_h = \begin{bmatrix}\delta \mathbf{v}_{h} \\ \delta \mathbf{u}_{h}\end{bmatrix} = H \delta \mathbf{v}_{h},
\end{align}
so $\norm{\delta \mathbf{w}_h}_p \geq \nu_p(H)\norm{\delta \mathbf{v}_h}_p$, and by combining this with the assumption $\mu_p(F) < 0$, we can rewrite \eqref{eq:pnorm} as the following
\begin{align}
\mu_p(F_r)
& \leq \lim\limits_{h \rightarrow 0} \frac{\| \delta \mathbf{w}_{h} + h F \delta \mathbf{w}_{h}\|_p -\| \delta \mathbf{w}_{h}\|_p}{h \norm{\delta \mathbf{w}_{h}}_p}\nu_p(H)\nonumber\\
& \leq \mu_p(F) \nu_p(H)\label{eq:FrlFp}.
\end{align}
Q.E.D.
% \begin{align}
% \frac{\norm{\delta \mathbf{w}_{h}}_p}{\norm{\delta \mathbf{v}_{h}}_p}
% &= \frac{\left\|\begin{bmatrix}\mathbb{1}\\S\end{bmatrix} \delta \mathbf{v}_{h}\right\|_p}{\|\delta \mathbf{v}_{h}\|_p} \nonumber\\
% & \geq \min \limits_{\norm{v}=1}\left\|\begin{bmatrix}\mathbb{1}\\S\end{bmatrix} v\right\|_p \nonumber\\
% & = \nu_p(S)
% \end{align}
% \begin{align}
% \norm{\delta \mathbf{w}_{h}}_p^p 
% &= \norm{\delta \mathbf{v}_{h}}_p^p + \norm{\delta \mathbf{u}_{h}}_p^p \nonumber\\
% &= \norm{\delta \mathbf{v}_{h}}_p^p + \norm{S\delta \mathbf{v}_{h}}_p^p \nonumber\\
% & \geq (1+ \min \limits_{\norm{v}=1} \norm{Sv}_p^p)\norm{\delta \mathbf{v}_{h}}_p^p \nonumber\\
% & = \left(1+ \nu_p^p(S)\right)\norm{\delta \mathbf{v}_{h}}_p^p
% % & = \left(1+ \frac{1}{\|S^{-1}\|_p^p}\right)\norm{\delta \mathbf{v}_{h}}_p^p
% \end{align}

The below corollary of Theorem \ref{theo:p21} provides the explicit expression of $\nu_p(S)$ with $p=2$.
\begin{corollary} (\textbf{$2$ norm})  \label{eq:col1}
Assuming that all assumptions of Theorem \ref{theo:p21} are satisfied, the contraction rate associated with the reduced system can be bounded as below
\begin{equation}
\mu_2(F_r) \leq \mu_2(F) \sqrt{1 +\sigma_{\min}^2(S)}
\end{equation}
in which $\sigma_{\min}(S)$ denotes the minimal singular value of the matrix $S$.
\end{corollary}
The proof of Corollary \ref{eq:col1} follows from Theorem \ref{theo:p21} and note that for $p=2$, we have that $\nu_2(H)= \sqrt{1 + \min \limits_{\norm{v}=1} \|S v\|_2^2} = \sqrt{1+ \sigma_{\min}^2(S)}$.
\end{proof}

\subsection{Converse theorems} \label{sec:converse}

\begin{theorem} \label{theo:p22}
(\textbf{$2$ norm}) For a contracting system $\dot{\mathbf{x}} = \mathbf{f}(\mathbf{x},\mathbf{Y}(\mathbf{x}))$ and metric function $\theta(\mathbf{x},t)$ with $\mu_2(F_r) < 0$ and any $\epsilon > 0$ there exists an extended system $F$ characterized by the matrices $Q,R,\rho$ contracting with the contraction rate satisfying $\mu_2(F) \leq \mu_2(F_r)/(1+\epsilon)$.

\end{theorem}

\begin{proof} \label{proof:2norm2}
This theorem can be proven by explicit construction of the matrices $Q, R, \rho$, which ensures fast enough contraction of $F$. The matrix $\rho$ is chosen to be small enough, so that $\sigma_{\max}^2(S) \leq \epsilon$ and $R = \eta D^{-T} \rho^T \rho D^{-1} C  P^{-1}$, $Q = -\eta D^{-T}\rho^T$, where $\eta = -\mu_2(F_r)/(1+\sigma_{\max}^2(S) )$. This choice of $R$ and $Q$ ensures that the symmetric part of $F$ is block-diagonal, so the following inequality follows from \eqref{eq:FFr}:
\begin{align} \label{eq:npdF}
&\delta{\mathbf{w}}^T F\delta{\mathbf{w}}= \begin{bmatrix}  \theta\delta{\mathbf{x}} \\  \rho\delta{\mathbf{y}} \end{bmatrix}^T \begin{bmatrix} F_r + \eta S^T S & 0 \\
  0 & -\eta \mathbb{1}\end{bmatrix}\begin{bmatrix}  \theta\delta{\mathbf{x}} \\  \rho\delta{\mathbf{y}} \end{bmatrix}\nonumber \\
  & = \delta{\mathbf{v}}^T \left(F_r+\eta S^TS\right) \delta{\mathbf{v}} - \eta \|\delta{\mathbf{u}}\|^2 \nonumber\\
  & \leq \left(\mu_2(F_r) +  \eta \sigma_{\max}^2(S)\right)\|\delta{\mathbf{v}}\|^2- \eta \|\delta{\mathbf{u}}\|^2  \nonumber\\
  & = -\eta \|\delta{\mathbf{w}}\|^2 + \left(\eta + \mu_2(F_r) + \eta \sigma_{\max}^2(S)\right)\|\delta{\mathbf{v}}\|^2 \nonumber \\
  & = -\eta \|\delta{\mathbf{w}}\|^2.
\end{align}
Since this inequality holds true for any $\delta {\mathbf{w}}$, we conclude that $\mu_2(F) \leq -\eta = \mu_2(F_r)/(1+\epsilon)$.

\end{proof}

The counterpart of Theorem \ref{theo:p22} for $1$ norm and $\infty$ norm are presented below.

% \begin{theorem} \label{theo:p1inf1}
% (\textbf{$1$ norm and $\infty$ norm}) For the system $\dot{\mathbf{x}} = \mathbf{f}(\mathbf{x},\mathbf{Y}(\mathbf{x}))$ and metric function $\theta(\mathbf{x},t)$, and contracting extended system $F$ with $\mu_p(F) < 0$ for $p = 1, \infty$, characterized by the matrices $Q,R, \rho$ the following relation holds:
% \begin{equation} \label{eq:FrFp}
% 	\mu_p(F_r) \leq \mu_p(F)
% \end{equation}
 
% \end{theorem}

% \begin{proof}
% The proof follows \eqref{eq:pnorm} for general $p$ norm and noting that $\frac{\norm{\delta \mathbf{w}_{h}}_p}{\norm{\delta \mathbf{v}_{h}}_p} \geq 1$.
% \end{proof}
\begin{theorem} \label{theo:p1inf2}
(\textbf{$1$ norm and $\infty$ norm}) For a contracting system $\dot{\mathbf{x}} = \mathbf{f}(\mathbf{x},\mathbf{Y}(\mathbf{x}))$ and metric function $\theta(\mathbf{x},t)$ with $\mu_p(F_r) < 0$ and any $\epsilon > 0$ there exists an extended system $F$ characterized by the matrices $Q,R,\rho$ contracting with the contraction rate satisfying $\mu_p(F)  \leq \mu_p(F_r)(1- \epsilon)$ where $p = 1, \infty$.

%. Denote $\sigma = \mu_2(\theta ^{-T} C^T C \theta^{-1})$. If $\mu_2(F_r) \leq -\beta_r$ then there $\exists Q \in \mathbb{R}^{m \times m},R \in \mathbb{R}^{m \times n}, \beta \in \mathbb{R_{>0}}, \beta \leq \beta_r/(1+\sigma)$:  $\mu_2(F) \leq - \beta$
 
%  the following statements are equivalent for any point $\mathbf{x}$:
%  \begin{enumerate}
%      \item $\mu_2(F_r) \leq -\beta$
%      \item $\exists Q \in \mathbb{R}^{m \times m},R \in \mathbb{R}^{m \times n}:  \mu_2(F) \leq - \beta/(1 + \norm{D^{-1}C \theta^{-1}})$
%  \end{enumerate}
% {\color{blue}Contracting with different rates}
\end{theorem}

\begin{proof}

Similar to proof \ref{proof:2norm2} we need to construct an appropriate tuple of matrices $Q$, $R$, $\rho$. Below we only present $\infty$ norm, but $1$ norm can be considered in the same way. Choosing $R = 0$, $Q = \mu_\infty(F_r) \rho D^{-1}$, and metric $\rho$ small enough so that $\norm{S}_\infty \leq \epsilon$, leads to a diagonally dominant matrix $F$ below:
\begin{align}
F = \begin{bmatrix}
F_r & 0\\
\mu_\infty(F_r) S & \mu_\infty(F_r) \mathbb{1} 
\end{bmatrix},
\end{align}
then we have the following relation:
\begin{align}
\mu_\infty(F) &=\max\{\mu_\infty(F_r), \nonumber\\
& \qquad \max \limits_{n+1 \leq i \leq n+m} \{\mu_\infty(F_r)+ \sum_{j=1}^{n}| \mu_\infty(F_r) S_{ij}| \} \}  \nonumber\\
&\leq \max \{ \mu_\infty(F_r), \mu_\infty(F_r) + |\mu_\infty(F_r)| \norm{S}_{\infty}\} \nonumber\\
&= \mu_\infty(F_r)(1-\epsilon).
\end{align}
Q.E.D.

\end{proof}

% For a physical system represented in DAEs form the associated reduced system characterized by the reduced Jacobian is unique. Speaking of the contraction rate the characteristic of the reduced one may manifest that of the physical system. In contrast there is a family of extended systems which result in the same reduced one or mathematically a family of matrices has the same Schur complement. This implies the fact that different projections from the corresponding extended spaces to the reduced space or the original state space may have the same image. Each member of this family may have different transient behaviors and contract with a different rate. The converse theorems in section \ref{sec:converse} show that the extended systems contract with slower rates than that of the physical one. 

% \begin{figure}[!ht]
%     \centering
%     \includegraphics[width=1\columnwidth]{images/Contraction.eps}
% 	\caption{Main theorems summary}
%      \label{fig:relation}
% \end{figure}
% The relations among contraction analysis of DAE systems for both reduced and extended systems are summarized in the Figure \ref{fig:relation}.
\subsection{Relation to other works} \label{sec:relate}
In this section we first discuss the relation to linear stability of DAE systems. The DAE systems have been studied extensively under ``descriptor" forms \cite{Lewis1985,TAKABA1995, bond2009stable, wu1994stability,MASUBUCHI1997669} as well as singular systems in \cite{verghese1981generalized}. In fact if there exists a matrix $Z$ that satisfies the Lyapunov inequality \eqref{eq:ZJ} in section \ref{sec:conreg} with $J_\star$ at an equilibrium $\mathbf{z}_\star$ and $\beta = 0$, then the descriptor system is asymptotically stable. Theorem \ref{theo:p22} not only suggests that the existence of such matrix $Z$ is indeed the sufficient condition for linear stability, but also provides an explicit construction of the certificate. The relation between the two notions of contraction and linear stability is further discussed below.

Contraction analysis and linear stability are closely related for autonomous systems. As discussed in section \ref{sec:invcons} if there exists a stable equilibrium that lies within a ball-like invariant region inscribed in the contraction region, all the trajectories of the systems starting inside the ball will shrink towards each other and merge to the nominal trajectory associated with the equilibrium; hence, the system is linearly stable at such equilibrium. In other words, if the system is linear stable at a particular equilibrium, there exists a contraction region centered at the equilibrium. This is true for ODEs as observed in \cite{pablo}. The converse theorems \ref{theo:p22} and \ref{theo:p1inf2} provide an explicit construction for DAEs. 

From the contraction and linear stability comparability, the inner approximated contraction region constructed below in section \ref{sec:conreg} is indeed a robust linear stability region in the variable space associated with the nominal operating point. Speaking of robust linear stability region, any equilibrium, if it exists and lies in such region, is a linearly stable one. Moreover as motivated by applying contraction analysis to the power systems which can be represented in DAE form, incremental stability implies convergence and vice versa. For the distinctions between the two concepts, interested readers can refer to \cite{ruffer2013convergent, angeli2002lyapunov}.

% In the context of singularly perturbed systems, the partially contracting characteristic can be applied to analyze the individual sub-systems \cite{bousquet2015contraction}. 
Singularly perturbed systems are also related to DAE systems as the time constant $\epsilon \rightarrow 0$. \cite{del2013contraction, bousquet2015contraction} revisit some key results of singular perturbations using contraction tools, where the fast and slow sub-systems are assumed to be partially contracting. Our approach here, on the other hand, doesn't require the systems to be partially contracting in the algebraic variable $\mathbf{y}$. In comparison to the key theorems from \cite{del2013contraction, bousquet2015contraction}, our results provide explicit conditions on the Jacobian matrices that can be applied to any system. However, to our knowledge, neither of our previously reported results on contraction of singularly perturbed systems dominate each other.

With respect to the contraction condition, the condition introduced in \cite{del2011contraction} can be recovered under our framework when $\epsilon$ goes to $0$, as follows. Consider the standard singular perturbation system:
\begin{align} \label{eq:singper}
    \dot{x}=f(x,y,t), \, \epsilon \dot{y} = g(x,y,t),\,\epsilon \geq 0.
\end{align}

Assume that the system \eqref{eq:singper} is partially contracting in $x$ and in $y$ with respect to transformation metrics $\theta$ and $\rho$. For simplicity let the transformation metrics be constant. Let's consider a Lyapunov function $V = \norm {[ \theta  \delta{\mathbf{x}} \, \rho \delta \mathbf{y} ]^T}$ then the system \eqref{eq:singper} is contracting if the generalized Jacobian $F_{sing} = \begin{bmatrix} \theta A \theta^{-1}& \theta B \rho^{-1}\\\rho C \theta^{-1}/\epsilon& \rho D \rho^{-1}/\epsilon \end{bmatrix}$ has a uniform negative matrix measure. Therefore if we select $Q = \rho^{-T} /\epsilon$, $R = D^{-T} B^T$, then $F_{sing} = F$. This implies that the central theorem \ref{theo:p21} applies to the DAE system associated to the system \eqref{eq:singper}.

%Note that system \eqref{eq:singper} becomes a DAE one as $\epsilon \rightarrow 0$, and in the absence of a bifurcation the algebraic part is contracting in identity transformation. We shall rederive the contraction condition for system \eqref{eq:singper} in \cite{del2013contraction} which is desribed as below.
In \cite{lohmiller1998contraction}, a similar class of systems was analyzed, where the linear constraints were imposed on the differential systems. The key conclusion that contraction of the original unconstrained flow implies local contraction
of the constrained flow is consistent with our results. However, apart from being constructive, our results also don't directly follow from the observations made in \cite{lohmiller1998contraction} as in our case the contraction of the extended system is not restricted to the algebraic manifold. %Therefore our approach does not require an explicit form of the algebraic manifold which is normally not known for general dynamical systems. Moreover, by applying Theorem \ref{theo:p22} and selecting suitable auxiliary matrices, one may obtain a less strict contraction condition.

\section{Inner approximation of contraction region} \label{sec:conreg}
%\subsection{Inner approximation based on 2-norm} \label{sec:conre}

In this section we constraint ourself to the class of DAE systems which can be represented by quadratic equations in variables $\mathbf{z}$. As a result the Jacobian $J$ depends linearly on $\mathbf{z}$. As identifying the real contraction region is challenging and even impossible in many practical situations, we introduce a technique using LMI formulation for constructing an inner approximation of contraction region centered at a given equilibrium based on $2$ norm. The approximated region has its merits of determining transient stability in electrical systems as shown in the application section. 

%Another technique for $1$ norm and $\infty$ norm is introduced in section \ref{sec:extend} {\color{red} add $1$ and $\infty$ norms}. 

% Instead one can construct an approximated contraction region in $\mathbf{z}$ space around the equilibrium, $\mathbf{z}_\star$ with corresponding fixed metric $\theta_\star$ such that every point inside such region satisfies $\mu_p(F)|_{\theta = \theta_\star} \leq 0$.
\begin{proposition} \label{prop:F}
The DAE system \eqref{eq:dif} and \eqref{eq:alg} is contracting in the contraction region $\mathcal{C}_p$ if there exists transform $\theta(\mathbf{x},t)$ such that $\exists \beta >0$, $\forall \mathbf{x} \in \mathcal{C}_p$, $\mu_p(F) \leq -\beta$.
\end{proposition}
\begin{proof}
$\mu_p(F_r)$ is negative follows from Theorem \ref{theo:p21}. The definition of contraction region in \eqref{eq:Fr} then concludes the proof.
\end{proof}
An important application of the Proposition \ref{prop:F} is to construct the contraction region by analyzing the extended systems where LMI formulation can be used for $2$ norm. Since the generalized unreduced Jacobian matrix introduced in \eqref{eq:F} is not suitable for LMI formulation, we rewrite the Jacobian in more convenient form as below with a constant metric $\theta$, $R = \tilde{R} + D^{-T} B^T$ and $Q = \tilde{Q}$, and $\rho = \mathbb{1}$ so that the system contracts in $\mathbf{y}$ with respect to the identity metric:
\begin{align}
\delta{\mathbf{w}}^T F\delta{\mathbf{w}} 
&= \begin{bmatrix}  \delta{\mathbf{x}} \\  \delta{\mathbf{y}} \end{bmatrix}^T \begin{bmatrix} P A + \tilde{R}^T & \tilde{R}^T D + P B  \\ \tilde{Q}^T C & \tilde{Q}^T D  \end{bmatrix}\begin{bmatrix}  \delta{\mathbf{x}} \\ \delta{\mathbf{y}} \end{bmatrix}\nonumber \\
 &= \delta{\mathbf{z}}^T Z^T J(\mathbf{z}) \delta{\mathbf{z}},
\end{align}
with a lower block diagonal auxiliary matrix $Z = \begin{bmatrix}
P & 0\\
\tilde{R} & \tilde{Q}
\end{bmatrix}$. Such matrix $Z$ also is used in linear stability assessment for descriptor systems. This is discussed in details in section \ref{sec:relate}. With the new representation, the problem of solving $\mu_2(F) \leq -\beta$ reduces to the following Lyapunov bilinear inequality in $Z$ and $J(\mathbf{z})$:
\begin{align} \label{eq:ZJ}
Z^T J(\mathbf{z})+ J(\mathbf{z})^T Z \preceq -\beta \mathbb{1}.
\end{align}
For any fixed $Z$ the equation \eqref{eq:ZJ} defines a spectrahedral region where the system is provably contracting. The invariant region around equilibrium point could have been constructed by inscribing the ball (ellipsoid) associated with the metric $\theta, \rho$ inside this region. However, inscription of ellipsoids inside spectrahedra is a NP-hard problem not scalable to the large power systems. Instead, we propose an alternative procedure, where we construct an intermediate polytopic region inscribed in a spectrahedron in which we inscribe the contracting ball. 

Due to the bilinear nature of \eqref{eq:ZJ} a contraction region can be characterized iteratively, for example through the following $2$-step procedure. First for a fixed point, say the equilibrium $\mathbf{z}_\star$ which without loss of generality can be assumed zero, one sets $J(\mathbf{z}) = J_\star$ then solves \eqref{eq:ZJ} for $Z$. Let $Z_\star$ be the solution of the first step, then the metric $\theta$ can be determined as the Cholesky decomposition of $P$. Next we fix $Z = Z_\star$ and perturb the system around its equilibrium $\mathbf{z}_\star$. The perturbation results in a linear Jacobian matrix $J(\mathbf{z}) = J_\star + \sum_k z_k J_k$ where some $J_k$ may vanish, which implies the fact that not all states and variables contribute to the Jacobian matrix. As long as the inequality \eqref{eq:ZJ} is satisfied, the maximum admissible variations $\mathbf{z}$ quantify the inner approximation of the contraction region. Expressly the goal of the second step is to identify a region in the variable space in which $Z_\star$ is the common matrix satisfying \eqref{eq:ZJ} for all inner points. This is equivalent to finding variations $\mathbf{z}$ satisfying the following LMI:
\begin{align} \label{eq:step2}
J(\mathbf{z})^T Z_\star + Z_\star^T J(\mathbf{z}) + \beta \mathbb{1} \preceq 0.
\end{align}
Note that \eqref{eq:step2} holds at the equilibrium. This leads to the following:
\begin{align}
Z_\star^T J_\star+ J_\star^T Z_\star + \beta \mathbb{1} = -U^T U \preceq 0.
\end{align}
Moreover for non-singular $U$ one can symmetrize \eqref{eq:step2} by multiplying on the left and the right by $U^{-T}$ and $U^{-1}$, respectively. As a result we have that
\begin{align} \label{eq:step2p}
-1 + \sum_k U^{-T}z_k(Z_\star^T J_k+ J_k^T Z_\star )U^{-1} \preceq 0.
\end{align}
For each coefficient matrix $J_k$, using SVD decomposition, yields the following: 
\begin{equation}
U^{-T}(Z_\star^T J_k+ J_k^T Z_\star )U^{-1}= \sum_h \lambda_{kh} e_{kh} e_{kh}^T.
\end{equation}
Then it's sufficient to conclude that \eqref{eq:step2p} holds if the following condition satisfies:
\begin{equation} \label{eq:newcriterion}
    \sigma_{max}\left(\sum_kz_k \sum_h \lambda_{kh} e_{kh} e_{kh}^T\right) \leq 1,
\end{equation}
which then can be formulated as the following:
\begin{align} \label{eq:opt}
& \max\limits_{\norm{v}_2 =1} \sum_k z_k \sum_h\lambda_{kh} (v^T e_{kh})^2  \nonumber\\
&\leq \max\limits_{k,h}\left[|z_k \lambda_{kh}| \sigma_{max}\left(\sum_{l,h} e_{lh} e_{lh}^T\right)\right]  \nonumber\\
&\leq 1.
\end{align}
% in which $H_k = \sum_{h=1}^{r_k} \lambda_{kh} e_{kh} e_{kh}^T$ using SVD decomposition with $r_k = rank(H_k) \leq N$.
% % Now we use that fact that $\sum_k a_kb_k \leq (\max\limits_k |a_k|) \sum_k |b_k|$ to solve the optimization problem \eqref{eq:opt} as well as to find the bounds of $z_k$ as below
% An upper-bound of the objective function in the above formulation can be given as
%     \begin{align} \label{eq:sigma}
%         \Sigma_{max} & \leq (\max\limits_{1 \leq k \leq l,1 \leq h \leq r_k}|z_k \lambda_{kh}|) \sigma_{max}(\sum_{k=1}^l \sum_{h=1}^{r_k} (e_{kh})^2)
% %         (\max\limits_{k,h}|z_k \lambda_{kh}|) \max\limits_v (\frac{\sum_k\sum_h (v^T e_{kh})^2 }{v^T v}) \nonumber\\
% %         & = (\max\limits_{k,h}|z_k \lambda_{kh}|) \sigma_{max}(\sum_k\sum_h (e_{kh})^2)
%     \end{align}
% Then the inequality in \eqref{eq:newcriterion} becomes:
% \begin{equation}
%     \max\limits_{1 \leq k \leq l,1 \leq h \leq r_k}|z_k \lambda_{kh}| \leq \frac{1}{\sigma_{max}(\sum_{k=1}^l\sum_{h=1}^{r_k} ( e_{kh})^2}
% \end{equation}

A box-type bound of the variation $z_k$ can be estimated as
% \begin{equation} \label{eq:expbound}
%     \max|z_k| \leq \frac{1}{ \sigma_{max}(\sum_k\sum_h  \lambda_{kh}e_{kh}e_{kh}^T)}
% \end{equation}
\begin{equation} \label{eq:expbound}
    \max|z_k| \leq \frac{1}{(\max_h |\lambda_{kh}|) \sigma_{max}(\sum_{l,h} e_{lh}e_{lh}^T)},
\end{equation}
which defines the bounds on each variable variation $z_k$ thus identifying the inner approximated contraction region. It can be seen that the explicit bound defined by \eqref{eq:expbound} depends on the uniform upper-bound of the contraction rate $\beta$ and matrix $Z_\star$. The bounds become more conservative for a large value of $\beta$. Since there is an infinite number of choices of $Z_\star$ satisfying $\mu_2(Z_\star ^T J_\star) \leq -\beta$, it's essential to understand which $Z_\star$ would correspond to the least conservative bounds. The bounds obtained from \eqref{eq:expbound} can be also improved by deploying a better estimated upper bound in \eqref{eq:opt}.

A similar inner approximated contraction region can be also constructed based on $1$ norm and $\infty$ norm. As many practical engineering systems require all physical quantities to be in compliance with specified operational constraints, we propose to use box constraint construction. We shall inscribe a box inside the spectrahedron by allowing some variability to the coordinates
\begin{equation}\label{eq:box}
    \underline{z_k} \leq z_k \leq \overline{z_k},
\end{equation}
in all the directions $k = 1,\dots n+m$.
To certify the box region indeed lies in the contraction region we formulate a robust linear programming showing that for all $\mathbf{z}$ satisfying \eqref{eq:box}, we have that $\mu_p(Z_\star^T J(\mathbf{z})) \leq -\beta$ where $Z_\star$ satisfies $\mu_p(Z_\star^T J_\star) \leq -\beta$. Moreover one can carry on the same procedure to check whether a ball-like region in the transformed space is inscribed in the contraction region.

\subsection{Invariant set construction} \label{sec:invcons}
In this section we describe the procedure for constructing an invariant set $\mathcal{I}_{2}$ that lies in the contraction region $\mathcal{C}_{2}$.

Assume that an inner approximation of the contraction region, $\mathcal{C}_{2}$, constructed from section \ref{sec:conreg} is a convex region defined by a set of linear inequalities $e_i^T \mathbf{x} \leq b_i$, for $i = 1,\dots,2 (n+m)$, and the equilibrium $\mathbf{x_\star} = 0$. $e_i$ is a unit vector with the non-zero element either $+1$ or $-1$ due to the box-type constraints. $b_i > 0$ represents the bound on the variation along each direction, and $b_i$ is set to infinity if the corresponding $G_k$ vanishes. The linear transformation with metric $\theta$ prompts a corresponding contraction region in $\mathbf{v}$ space, i.e. $\mathcal{C}_{2\,\mathbf{v}} = \{\mathbf{v}| e_i^T \theta^{-1} \mathbf{v} \leq b_i \}$, for $i = 1,\dots,2 (n+m)$. Then to construct an invariant set we find the largest Euclidean ball centered at the equilibrium where $\mathbf{v}_\star = \theta \mathbf{x}_\star = 0$ that lies in $\mathcal{C}_{2\,\mathbf{v}}$. The problem can be formulated as the following LP:
\begin{align} \label{eq:invset}
&\underset{r}{\text{maximize}} \quad r \nonumber \\
& \text{subject to} \quad q_i(r) \leq b_i; \quad r \geq 0 \quad i = 1,\dots,2(n+m)
\end{align}
where the constraints be 
\begin{align} \label{eq:qi}
q_i(r) 
&= \sup\limits_{\norm{u} \leq 1} e_i^T \theta^{-1}(v_\star + r u) \nonumber \\
&= r \norm{e_i^T \theta^{-1}}_2.
\end{align}
\eqref{eq:qi} follows from the Cauchy-Schwarz inequality, i.e. for a nonzero vector $x$, the vector $u$ satisfying $\|u\|_2 \leq 1$ that maximizes $x^Tu$ is $x/\|x\|_2$. It also can be seen that the LP \eqref{eq:invset} admits the optimal solution $r_{max} = \min\limits_{i} \{ b_i / \norm{e_i^T \theta^{-1}}_2\}.$

%To show that the ball is an invariant set is simple as the following. Let the nominal trajectory correspond to the equilibrium. Since the ball is inscribed in the contraction region, all trajectories starting from inside the ball will remain within the ball and contract with respect to the nominal one; eventually they will converge to the equilibrium.

\section{Transient stability of power systems}

In this section we demonstrate how the developed techniques can be applied to the problem of constructing inner approximations of the contraction regions applicable to transient stability analysis of power systems modeled in DAE forms.
\subsection{Large-disturbance stability}

Large disturbance stability or transient stability is defined as the ability of the system to maintain synchronism after being subject to major disturbances such as line failures or loss of large generators or loads. Unstable systems will exhibit large angle separation or voltage depression which lead to system disintegration \cite{kundur2004definition}. The objective of transient analysis is to determine whether the system can converge to a feasible post-fault stable equilibrium for a given pre-fault stable operating point and a trajectory along which the system evolves during the fault, the so-called fault-on trajectory. Assuming that all operational constraints or feasibility conditions, and stability conditions are satisfied at the post-fault equilibrium, we then go into the convergence of the post-fault trajectory.

There are two main approaches to transient stability analysis including time-domain simulations and energy based or direct methods \cite{ChiangBook}. An alternative based on inner approximated contraction region is then proposed. This doesn't require intensive computation efforts like the time-domain approach while still providing a reasonably non-conservative stability region in the state space. As long as the initial point of the post-fault trajectory lies inside such region, the convergence to post-fault stable equilibrium is guaranteed.

The salient features of the contraction approach include scalibility, online analysis facilitation, and it does not require tailored energy function construction. The heaviest computational tasks are solving Lyapunov inequalities and SVD decomposition, which even of large scale problems, are ready to be solved with existing algorithms in regular processors. The contraction approach also allows the analysis be free from post-fault trajectory numerical integration which is time consuming and prevents online assessment. The third feature makes a key distinction between the contraction approach and the direct methods. Indeed the direct methods rely on energy function construction which doesn't have a general form in lossy networks and there is a need for finding critical energy levels based on which a stability region is identified. The contraction approach, on the other hand, just requires the transform $\theta$ and $\rho$ under which the system is contracting. Once the transform is found through solving Lyapunov inequalities a corresponding sub-region of the contraction region can be constructed.

It's worth mentioning that the inner approximated contraction region is also a robust linear stability region so that the post-fault equilibrium is stable if it is an interior point. By construction the feasibility of the constructed region is easily validated as well. More importantly based on the inner approximated region one can either gain insight about the system stability ``degree'' or preliminarily compute ``sufficient'' critical clearing time ($sCCT$) which is more strict than the actual $CCT$, the maximum allowed fault-on duration. Hence if the fault is cleared before $sCCT$, the fault-on trajectory won't escape or exit the invariant region and the post-fault trajectory will converge to a post-fault stable equilibrium inside the invariant region. For more details on $CCT$ one can refer to \cite{ChiangBook,Kundur}.
%and multiple swings aptness.
\subsection{A $2$-bus system}
The applications of approximating the contraction region are discussed above. In this section we illustrate the procedure by constructing one for a two-bus system as shown in Figure \ref{fig:rudi} \cite{Venkatasubramanian1992}.
\begin{figure}[!ht]
    \centering
    \includegraphics[width=0.8\columnwidth]{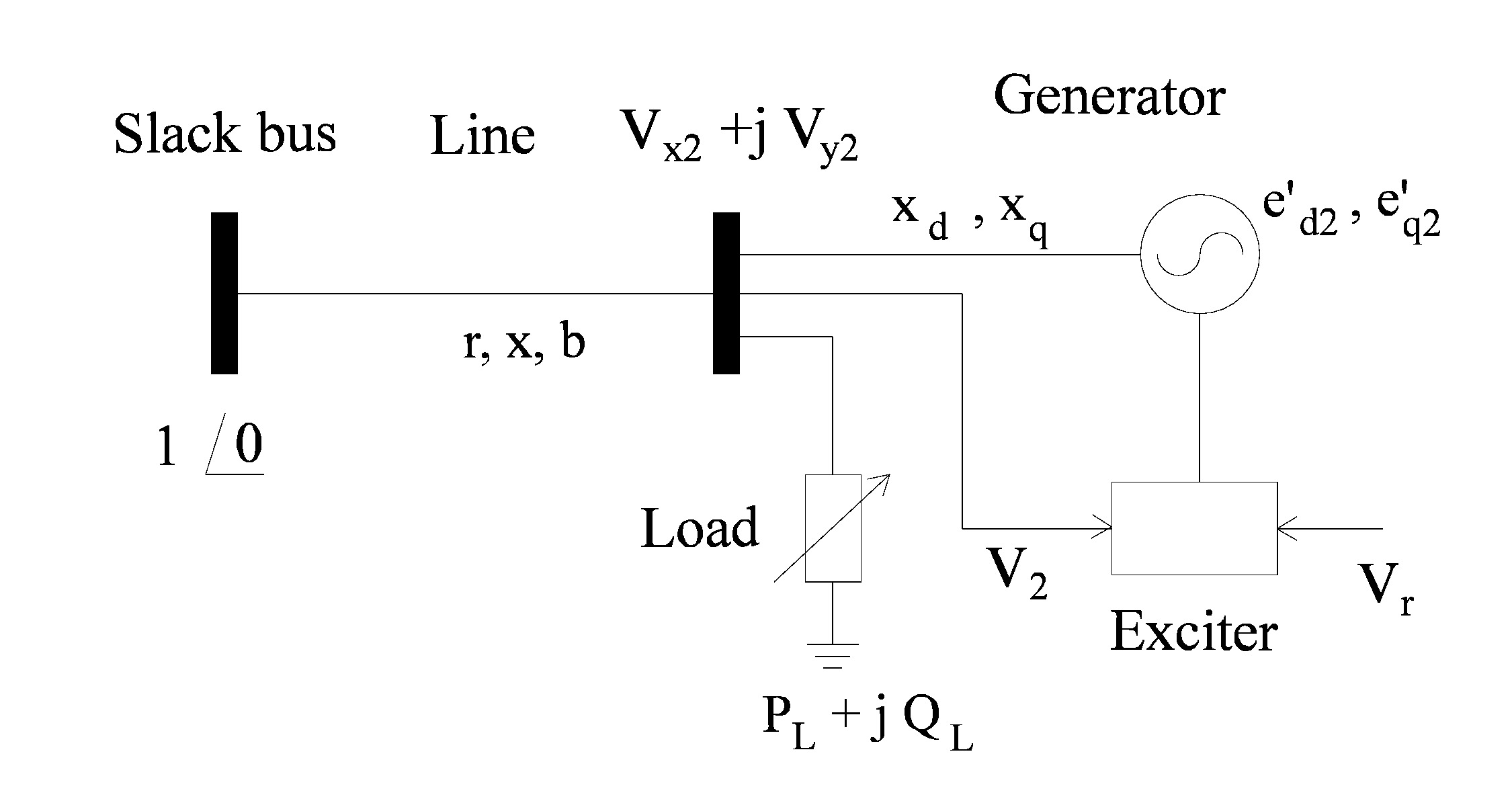}
	\caption{A $2$-bus system}
    \label{fig:rudi}
\end{figure}
The $2$-bus system includes one slack bus, and one generator bus with a load residing at the same bus. The slack bus voltage is specified, i.e. $V_1 = 1.04\angle 0$. The generator, modeled with a high order generator model, maintains the voltage at bus $2$ and generates active power at specific levels, i.e. $V_2 = 1.025\,p.u.$ and $P_G = 0.8\,p.u.$. The load consumes fixed amount of powers, $P_L = 1.63\,p.u.$ and $Q_L=1.025\,p.u.$. Note that hereafter we use $r$ to denote the line resistance.

The sets of differential equations $\dot{\mathbf{x}} = f(\mathbf{x},\mathbf{y})$ which describe the dynamics of the generator are listed below. The details are introduced in \cite{PSAT}.
\begin{align}
T'_{d0} \frac{d}{dt}{e}'_{q2} &= -e'_{q2} - (x_d - x'_d-\frac{T''_{d0}}{T'_{d0}}\frac{x''_{d}}{x'_d}(x_d-x'_d)) i_{d2} \nonumber\\
&\quad +(1-\frac{T_{AA}}{T'_{d0}}V_2), \nonumber\\
T''_{d0} \frac{d}{dt}{e}''_{q2} &= - e''_{q2} + e'_{q2} - (x'_d - x''_d-\frac{T''_{d0}}{T'_{d0}}\frac{x''_{d}}{x'_d}(x_d-x'_d)) i_{d2} \nonumber\\
& \quad + \frac{T_{AA}}{T'_{d0}}V_2,\nonumber\\
T'_{q0} \frac{d}{dt}{e}'_{d2} &= -e'_{d2} + (x_q - x'_q-\frac{T''_{q0}}{T'_{q0}}\frac{x''_{q}}{x'_q}(x_q-x'_q)) i_{q2},\nonumber\\
T''_{q0} \frac{d}{dt}{e}''_{d2} &= -e''_{d2} + e'_{d2} + (x'_q - x''_q-\frac{T''_{q0}}{T'_{q0}}\frac{x''_{q}}{x'_q}(x_q-x'_q)) i_{q2},\nonumber
\end{align}
\begin{align}
\frac{d}{dt}{\sin\delta'_2} &= 2\pi f_n \cos\delta'_2(-1 + \omega_2),\nonumber\\
M \frac{d}{dt}{\omega}_2 &= p_m - i_{d2} v_{d2} - i_{q2} v_{q2} - D(-1 + \omega_2).
\end{align}

Algebraic equations, $g(\mathbf{x},\mathbf{y}) = 0$, are composed of the relations describing the generator, the network, and the load, that can be stated as follow:
\begin{align}
0 &= -e''_{q2} + x''_{d2} i_{d2}+ v_{q2}, \nonumber\\
0 &= -e''_{d2} - x''_{q2} i_{q2}+ v_{d2}, \nonumber\\
0 &= -v_{q2} + \cos\delta'_2 v_{x2} + \sin\delta'_2 v_{y2},\nonumber\\
0 &= -v_{d2} + \sin\delta'_2 v_{x2} - \cos\delta'_2 v_{y2},\nonumber\\
0 &= (\cos\delta'_2)^2 + (\sin\delta'_2)^2 -1,\nonumber\\
0 &= \cos\delta'_2 i_{q2} + i_{d2} \sin\delta'_2 + \frac{b v_{y2}}{2} + \frac{r V_1}{r^2+x^2} \nonumber\\
& \quad-\frac{x v_{y2}}{r^2+x^2} -P_L v_{x2}-Q_L v_{y2},\nonumber\\
0 &= -\cos\delta'_2 i_{d2} + i_{q2} \sin\delta'_2 -\frac{r v_{y2}}{r^2+x^2}\nonumber\\
& \quad -\frac{r V_1}{r^2+x^2} + Q_L v_{x2}-P_L v_{y2},\nonumber\\
0 &= v_2^2 - v_{x2}^2 - v_{y2}^2.
\end{align}

For the $2$-bus system, the set of variables includes $6$ states, $\mathbf{x} = [E_{q2}', E_{q2}'', E_{d2}', E_{d2}'',\sin(\delta_2'), \omega_2]^T$, and $8$ algebraic variables, $\mathbf{y}= [i_{d2}, i_{q2}, V_{d2}, V_{q2}, V_2, V_{x2}, V_{y2},\cos(\delta_2')]^T$, where the subscript $2$ indicates bus number $2$.
The system parameters are given as the following: $T'_{d0}=0.6$, $T''_{d0}=0.02$, $x_q = 0.8958$, $x'_q = 0.1969$, $x''_q = 0.1$, $T'_{q0} = 0.535$, $T''_{q0} = 0.02$, $M = 12.8$, $D=20$, $T_{AA} = 0.002$, $p_m = P_G$, $r = 0.01938$, $x = 0.05$, $b = 0.0528$, $f_n = 60$. All parameters are in $p.u.$ except time constants in seconds and frequency in Hertz.
%\subsubsection{Numerical simulations}

A dynamic simulation and analysis package is developed in Mathematica $10.3.0.0$ taking PSAT dynamic models \cite{PSAT} as the input. We also use CVX in MATLAB for solving Lyapunov inequalities.

\begin{figure}[!ht]
    \centering
    \includegraphics[width=0.8 \columnwidth]{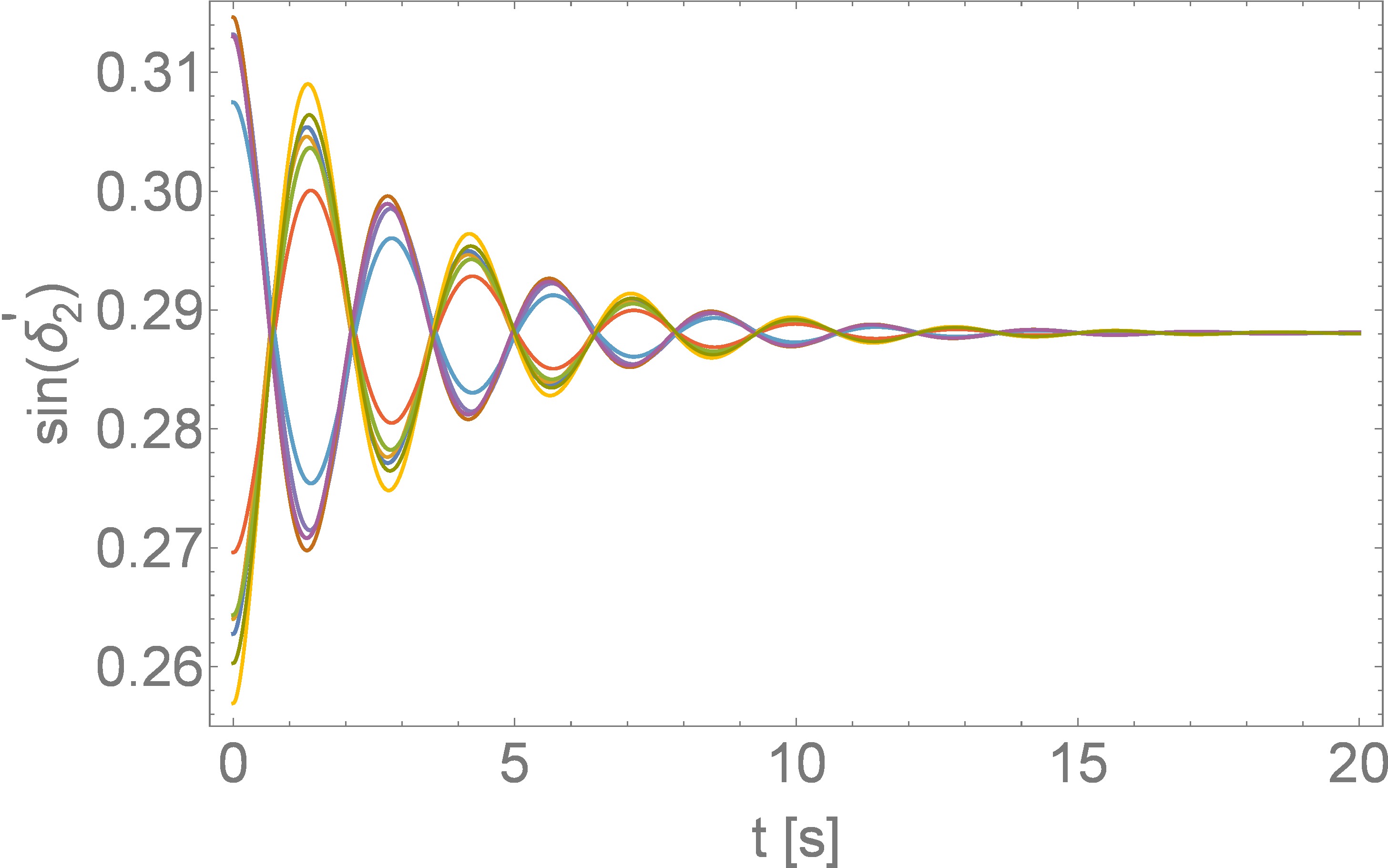}
	\caption{The state $\sin{(\delta_2')}$ of the generator simulated to $20\,s$}
     \label{fig:sindel}
\end{figure}
%with $r = 0.057$
% \begin{figure}[!ht]
%     \centering
%     \includegraphics[width=1\columnwidth]{images/vd2.eps}
% 	\caption{The $v_{d,2}$ of the generator with $r_\beta = 0.24$, simulated to $100\,s$}
%      \label{fig:vd2}
% \end{figure}
As shown in Figure \ref{fig:sindel}, the system will converge to the nominal equilibrium if the gaps between the initial values of states, i.e. $\sin{\delta'_2}$ in this case, and the nominal values do not exceed the maximum distance that corresponds to the maximum radius $r_{max}$ and the metric $\theta$ as discussed in the invariant set construction in section \ref{sec:invcons}.

Figure \ref{fig:largestball} shows the contraction region in state space which is an ellipsoidal region corresponding to the ball-like invariant set as discussed in section \ref{sec:conreg}. The convergence of all inner trajectories confirms that if the system starts from inside the ball, the corresponding trajectory is contained at all times. This can be interpreted as the following: if the post-fault equilibrium and the initial point of the post-fault trajectory both are a part of the region inside the ball, the system is transient stable. It also can be seen that the constructed invariant region touches the approximated contraction region boundary which associates with the state $\sin{\delta'_2}$. By assigning non-uniform weights to variables $\mathbf{z}$ in \eqref{eq:opt}, the invariant region can be stretched along other directions as well.
% The attraction region is a largest ball in $\mathbf{v}$ space or its image, an ellipsoid in $\mathbf{x}$ space as shown in Figure \ref{fig:largestball} which confirms that 

\begin{figure}[!ht]
    \centering
    \includegraphics[width=1\columnwidth]{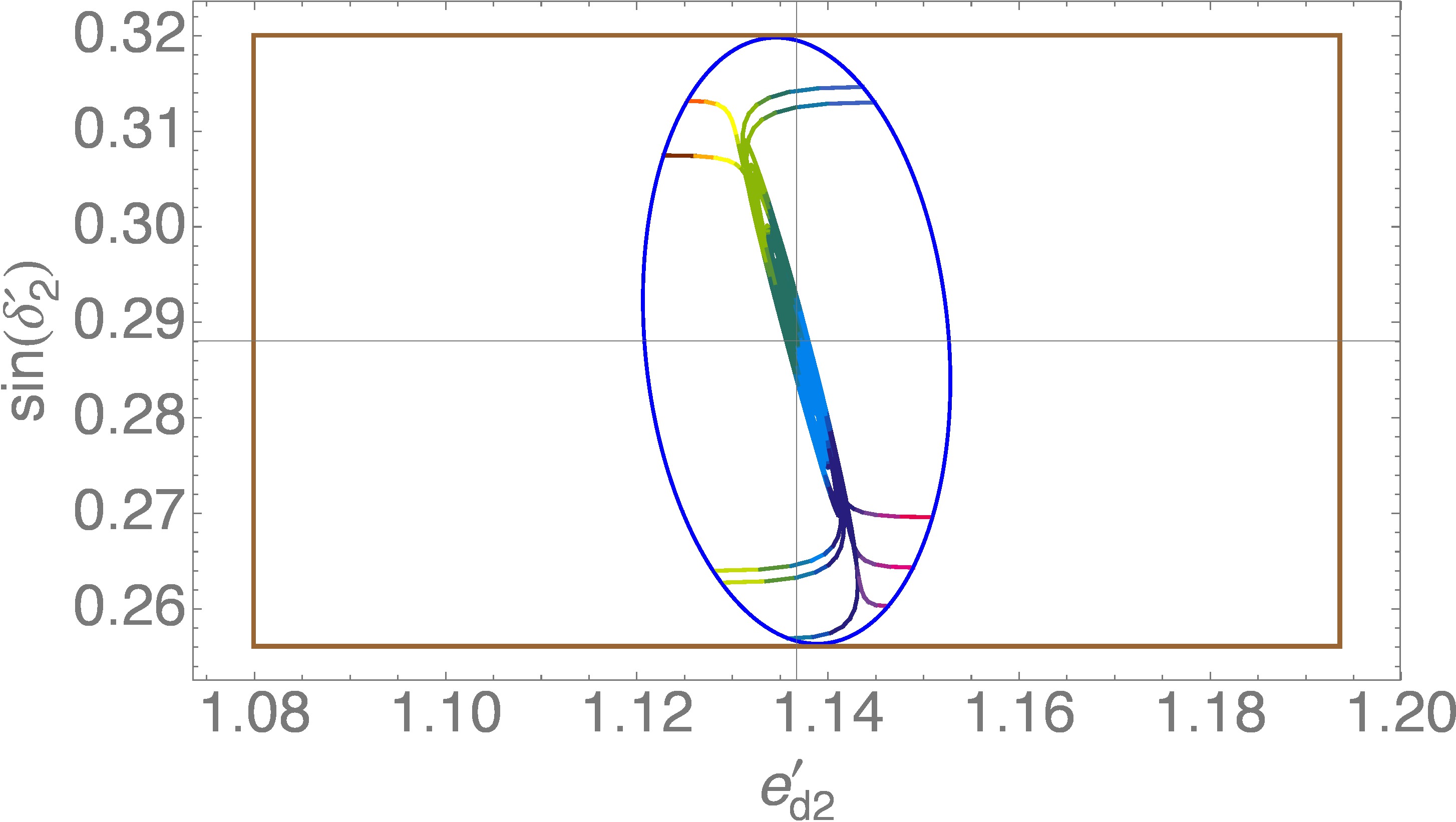}
	\caption{The ellipsoidal invariant region
    }
     \label{fig:largestball}
\end{figure}

\section{Conclusion}
In this work the contraction properties of Differential-Algebraic Systems were characterized in terms of the contracting properties of the extended Jacobian representing the virtual differential system that reduces to a given DAE under singular perturbation theory analysis. We established the relations between the contraction rates of the extended ODE and reduced DAE systems and used these relations to develop a systematic technique for constructing inner approximations of the attraction region for quadratic DAE systems. 

%The paper shows that there is an equivalent relation between the original systems and the reduced ones in contraction analysis based on most common norms. We also propose scalable techniques for constructing inner approximation of contraction regions based on the three considered norms and inscribing a ball-like invariant region. The application of the constructed invariant region is demonstrated in the context of transient stability analysis for power systems.

In the future we plan to extend our results to develop a more accurate characterization of the contraction of systems with strong time-scale separation and explore how the framework can be used for systematic decomposition of complex and large scale systems for distributed control/analysis purposes.

\section*{Acknowledgment}
This work was supported in part by NSF-ECCS award ID 1508666, MIT/Skoltech and Masdar Initiative, Ministry of Education and Science of Russian Federation, Grant agreement No. 14.615.21.0001, Grant identification code: RFMEFI61514X0001, as well as the Vietnam Education Foundation.

\appendices

\bibliographystyle{IEEEtran}

\bibliography{main.bbl}

%  \begin{IEEEbiography}
%  [{\includegraphics[width=1in,height=1.25in,clip,keepaspectratio]{mshell}}]{Michael Shell}

\begin{IEEEbiography} [{\includegraphics[width=1in,height=1.25in,clip,keepaspectratio]{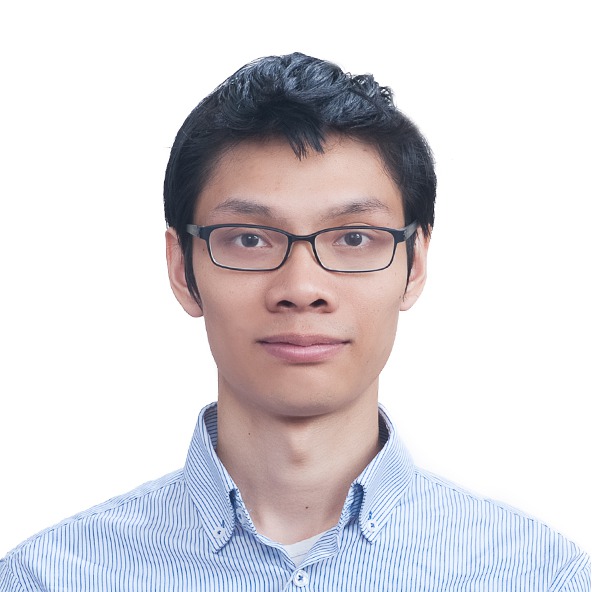}}]{Hung D. Nguyen}
(S`12) was born in Vietnam, in 1986. He received the B.E. degree in electrical engineering from Hanoi University of Technology, Vietnam, in 2009, and the M.S. degree in electrical engineering from Seoul National University, Korea, in 2013. He is a Ph.D. candidate in the Department of Mechanical Engineering at Massachusetts Institute of Technology (MIT). His current research interests include power system operation and control; the nonlinearity, dynamics and stability of large scale power systems; DSA/EMS and smart grids.
\end{IEEEbiography}
\vspace{-20 mm}
\begin{IEEEbiography} [{\includegraphics[width=1in,height=1.25in,clip,keepaspectratio]{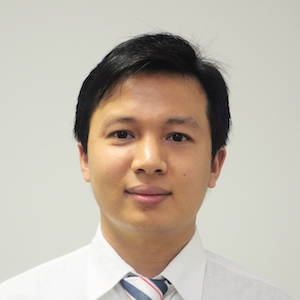}}]{Thanh Long Vu} (M`13) received the B.Eng. degree in automatic
control from Hanoi University of Technology in 2007 and the Ph.D. degree in electrical engineering from National University of Singapore in 2013.
Currently, he is a Research Scientist at the Mechanical Engineering Department of Massachusetts Institute of Technology (MIT). Before joining MIT, he was a Research Fellow at Nanyang Technological University, Singapore. His main research interests lie at the intersections of electrical power systems, systems theory, and optimization. He is currently
interested in exploring robust and computationally tractable approaches for
risk assessment, control, management, and design of large-scale complex
systems with emphasis on next-generation power grids
\end{IEEEbiography}
\vspace{-20 mm}
\begin{IEEEbiography} [{\includegraphics[width=1in,height=1.25in,clip,keepaspectratio]{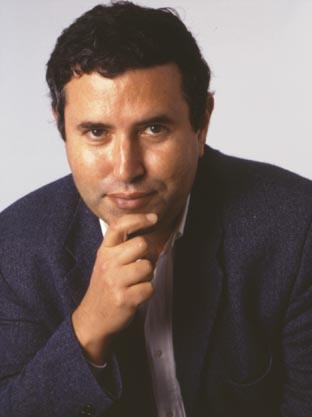}}]{Jean-Jacques Slotine}
(M`83) was born in Paris in 1959, and received his Ph.D. from the Massachusetts Institute of Technology in 1983. After working at Bell Labs in the computer research department, in 1984 he joined the faculty at MIT, where he is now Professor of Mechanical Engineering and Information Sciences, Professor of Brain and Cognitive Sciences, and Director of the Nonlinear Systems Laboratory. He is the co-author of the textbooks “Robot Analysis and Control” (Wiley, 1986) and “Applied Nonlinear Control” (Prentice-Hall, 1991). Prof. Slotine was a member of the French National Science Council from 1997 to 2002, and of Singapore's A*STAR SigN Advisory Board from 2007 to 2010. He is currently on the Scientific Advisory Board of the Italian Institute of Technology.
\end{IEEEbiography}
\vspace{-20 mm}
\begin{IEEEbiography} [{\includegraphics[width=1in,height=1.25in,clip,keepaspectratio]{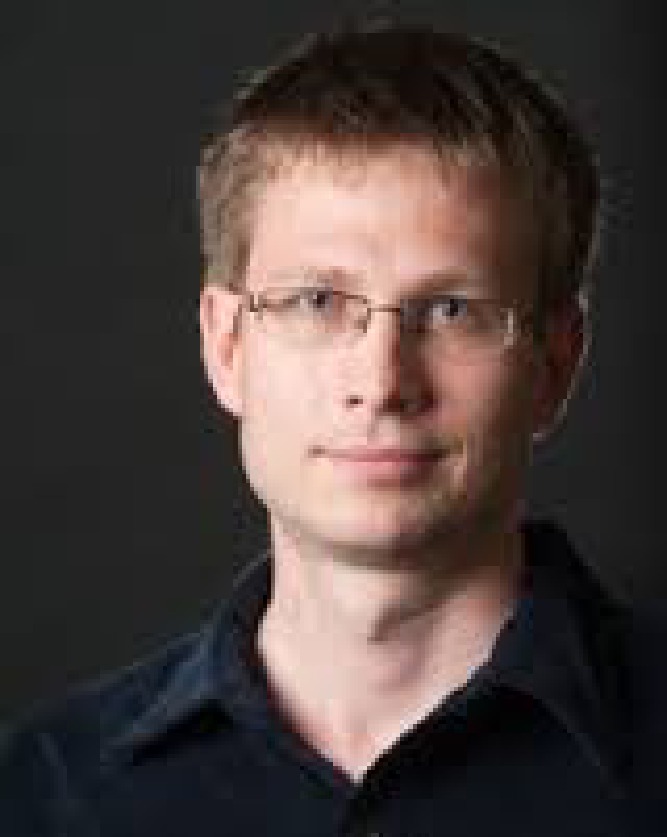}}]{Konstatin Turitsyn}
(M`09) received the M.Sc. degree in physics from Moscow Institute of Physics and Technology and the Ph.D. degree in physics from Landau Institute for Theoretical Physics, Moscow, in 2007.  Currently, he is an Associate Professor at the Mechanical Engineering Department of Massachusetts Institute of Technology (MIT), Cambridge. Before joining MIT, he held the position of Oppenheimer fellow at Los Alamos National Laboratory, and Kadanoff-Rice Postdoctoral Scholar at University of Chicago. His research interests encompass a broad range of problems involving nonlinear and stochastic dynamics of complex systems. Specific interests in energy related fields include stability and security assessment, integration of distributed and renewable generation.
\end{IEEEbiography}

\end{document}